\documentclass[10 pt, leqno]{article}
\date{}

\pdfoutput=1

\usepackage{amssymb, amsbsy, amsmath, amsfonts, amssymb, amscd, mathrsfs, amsthm, dsfont,enumerate}%,ulem}

\usepackage[english]{babel}
\usepackage{indentfirst}

\makeatletter
\@addtoreset{equation}{section}
\makeatother

\newtheorem{statement}{}[section]
\newtheorem{theorem}[statement]{Theorem}
\newtheorem{lemma}[statement]{Lemma}

\newtheorem{example}[statement]{Example}
\newtheorem{proposition}[statement]{Proposition}
\newtheorem{definition}[statement]{Definition}
\newtheorem{corollary}[statement]{Corollary}

\newcommand\C{\mathbb C}

\newcommand\R{\mathbb R}

\newcommand\D{\mathbb D}

\newcommand\e{{\rm e}}

\newcommand\ind{\mathds{1}}
\newcommand\dis{\displaystyle}

\newcommand\converge{\mathop{\longrightarrow}\limits}

\newcommand{\biindice}[3]%
{%

\begin{array}[t]{c}
{\displaystyle #1}\\
{\scriptstyle #2}\\
{\scriptstyle #3}
\end{array}

}

\let\phi=\varphi
\let\hat = \widehat
\let\tilde=\widetilde

\title{\bf On some questions about composition operators on weighted Hardy spaces}

\author{\it Pascal~Lef\`evre, Daniel~Li,  \\ \it Herv\'e~Queff\'elec, Luis~Rodr{\'\i}guez-Piazza}

\date{{Dedicated to Gilles Godefroy for his $70^{th}$ birthday} \\ 
\bigskip \footnotesize \today}

\begin{document}

\maketitle

\noindent {\bf Abstract.} We first consider some questions raised by N.~Zorboska in her thesis. In particular she asked for which sequences $\beta$ every 
symbol $\varphi \colon \D \to \D$ with $\varphi \in H^2 (\beta)$ induces a bounded composition operator $C_\phi$ on the weighted Hardy space $H^2 (\beta)$. 
We give partial answers and investigate when $H^2 (\beta)$ is an algebra.  
We answer negatively another question in showing that there are a sequence $\beta$ and $\varphi \in H^2 (\beta)$ such that $\| \varphi \|_\infty < 1$ and the 
composition operator $C_\varphi$ is not bounded on $H^2 (\beta)$. 

In a second part, we show that for $p \neq 2$, no automorphism of $\D$, except those that fix $0$, induces a bounded composition operator on the 
Beurling-Sobolev space $\ell^p_A$, and even on any weighted version of this space. 
\medskip

\smallskip


%%%%%%%%%%%%%%%%%%%%%%%%%%%%%%%%%%%%%%%%%%%%%%%%%%%%%%%%%%%%%%%%%%%%%%%%%%%
\section {Introduction} \label{sec: intro}

Let $\beta = (\beta_n)_{n \geq 0}$ be a sequence of positive numbers such that 
\begin{equation} \label{condition on beta}
\liminf_{n \to \infty} \beta_n^{1 / n} \geq 1 \, . 
\end{equation}
The associated weighted Hardy space $H^2 (\beta)$ is the Hilbert space of analytic functions $f \colon \D \to \C$ such that 
\begin{equation}
\| f \|^2 := \sum_{n = 0}^\infty |a_n |^2 \beta_n < + \infty \, ,
\end{equation}
if $f (z) = \sum_{n = 0}^\infty a_n z^n$. We shall denote $a_n = \widehat{f}(n)$. 
Note that condition \eqref{condition on beta} ensures that those functions are indeed analytic on $\D$, and conversely. 
\smallskip

More generally, for $1 \leq p \leq \infty$, the weighted space $h^{p}(\beta)$ (called Beurling-Sobolev space in \cite{ENZ} and \cite{SZZA}) 
is the space of all analytic functions $f (z) = \sum_{n = 0}^\infty a_n z^n$ on the unit disk $\D$ such that $\| f \|_{\beta}$ is finite, where 
\begin{displaymath}
\Vert f \Vert_{\beta}^{p} := \sum_{n = 0}^\infty |a_n|^p \beta_n \, .
\end{displaymath}
for $1 \leq p < \infty$, and $\| f \|_\infty = \sup_{n \geq 0} |a_n|\, \beta_n$. For $\beta \equiv 1$, we write simply $h^p$ or $\ell^p_A$. 

Observe that $h^2$ is none other than the usual Hardy space $H^2$, and that $h^{2} (\beta)  = H^{2} (\beta)$. However, for $p \neq 2$, 
$h^p$ has nothing to do with the usual Hardy space $H^p$, even if, for example, $H^p \subseteq h^\infty$. But in the spirit of our recent work 
\cite{LLQR-weighted}, the case of $h^p$ seems a natural one to consider. 

Again, \eqref{condition on beta} is equivalent to the inclusion $h^{p} (\beta) \subseteq {\mathscr H} (\D)$ and allows to treat the elements of 
$h^{p} (\beta)$ as analytic functions on $\D$.

Though most of our results hold for all $p \geq 1$, we first restrict ourselves to $p = 2$, and only consider the other exponents $p$ in Section~\ref{sec hp}. 
\medskip

If $\phi \colon \D \to \D$ is a non-constant analytic self-map of $\D$, we say that $\phi$ is a \emph{symbol}. To each symbol, we associate the 
\emph{composition operator} $C_\phi$ defined as $C_\phi (f) = f \circ \phi$ for every analytic function $f \colon \D \to \C$. 
\medskip

In her thesis, N.~Zorboska \cite{ZorboPhD} raised the following problems. 
\smallskip 

\begin{itemize}
\item [1)] Determine the spaces $H^2 (\beta)$ for which every symbol $\phi$, with $\phi \in H^2 (\beta)$, induces a bounded composition operator on 
$H^2 (\beta)$ (Problem~2). 
\smallskip

\item [2)] If $\phi \in H^2 (\beta)$ and $\| \phi \|_\infty < 1$, is the composition operator $C_\phi$ necessarily bounded on $H^2 (\beta)$? (Problem~3). 
\end{itemize}
\smallskip

In Section~\ref{answers}, we give partial answers to the first problem (Theorem~\ref{theo CN SO} and 
Theorem~\ref{partial answer}), and a negative one to the second problem (Theorem~\ref{cortheo not algebra}). 
\smallskip

In Section,~\ref{sec hp}, we show that for $p \neq 2$, no automorphism $T_a$ of $\D$ (to be defined right after), with $a \neq 0$, induces a 
bounded composition operator on the weighted Hardy space $h^p (\beta)$, whatever $\beta$. 
\medskip

After the first version of this paper was completed, we learned the existence of \cite{N1} and \cite{ENZ}. This version takes these papers into account. 
We thank O.~El-Fallah for sending us a copy of the english version of \cite{ENZ}. We thank also {N.~}Nikolski who learned us that 
\cite[\S\,3.2]{N4} deals with the question of weighted algebras, and that this question goes back to \cite{N2} (without proof), \cite{N3}, 
and \cite{N1} (without proof). 

%%%%%%%%%%%%%%%%%%%%%%%%%%%%%%%%%%%%%%%%%%%%%%%%%%%%%%%%%%%%%%%%%%%%%%%%%%%%%
\section{Answers to Zorboska's questions}  \label{answers}

\subsection{A necessary condition}

In \cite{LLQR-weighted}, we characterized the weights $\beta$ for which all composition operators are bounded on $H^2 (\beta)$: this happens if and only 
if $\beta$ is \emph{essentially decreasing} and \emph{slowly oscillating}.
\goodbreak 

\begin{definition}
A sequence of positive numbers $\beta = (\beta_n)_{n\geq 0}$ is said to be

$1)$ \emph{essentially decreasing} if, for some constant $C \geq 1$, we have
\begin{equation} 
\qquad \beta_m \leq C \, \beta_n \quad \text{for all } m \geq n \geq 0 \, ;
\end{equation} 

$2)$ \emph{slowly oscillating} if there are positive constants $c < 1 < C$ such that 
\begin{equation}
\qquad \quad c \, \beta_n \leq \beta_m \leq C \, \beta_n \quad \text{when } n / 2 \leq m \leq 2 n \, . 
\end{equation}
\end{definition}

Note that if $\beta$ is slowly oscillating, then $\liminf_{n \to \infty} \beta_n^{1 / n} = 1$.
\medskip

First, for the first Zorboska's question, we have the following necessary condition. 

\begin{theorem} \label{theo CN SO}
Assume that $\lim_{n \to \infty} \beta_n^{1 / n} = 1$. 

If every symbol $\phi \in H^2 (\beta)$ induces a bounded composition operator on $H^2 (\beta)$, then $\beta$ is slowly oscillating. 
\end{theorem}
\begin{proof}
For all $a \in \D$, we consider the automorphism $T_a$ defined by 
\begin{equation} \label{def T_a}
T_a (z) = \frac{a + z}{1 + \bar{a} z} \, \cdot
\end{equation}

Let us point out that for every integer $n \ge1$, we have 
$$\widehat{T_a} (n) = (- 1)^{n - 1} \bar{a}^{n - 1}(1 - |a|^2) \, .$$

Therefore, all the symbols $T_a$ are in $H^2 (\beta)$ if and only if for every $a\in\D$, we have $\sum\beta_n|a|^{2n}<\infty$, which is equivalent to the 
fact that the Taylor series $\sum\beta_n z^n$ has a radius of convergence at least $1$ i.e. $\dis\limsup \beta_n^{1 / n} \le 1$.
 
Since $\lim_{n \to \infty} \beta_n^{1 / n} = 1$, by hypothesis, the preceding remark implies that all the symbols $T_a$ define bounded composition operators 
on $H^2 (\beta)$, by the hypothesis. By \cite[Theorem~4.9]{LLQR-weighted}, it follows that $\beta$ is slowly oscillating. 
\end{proof}
%

%%%%%%%%%%%%%%%%%%%%%%%%%%%%%%%%%

\subsection {$H^2 (\beta)$ as an algebra} 

In order to motivate the content of this section, we point out the following easy fact. 

\begin{proposition} \label{prop easy}
Assume that every symbol $\phi \in H^2 (\beta)$ induces a bounded composition operator on $H^2 (\beta)$. Then $H^2 (\beta) \cap H^\infty$ is an algebra. 
\end{proposition}
\begin{proof}
It suffices to show that $f^2  \in H^2 (\beta) \cap H^\infty$ for every $f  \in H^2 (\beta) \cap H^\infty$. 
Let $f \in H^2 (\beta) \cap H^\infty$, not constant. If $M > \| f \|_\infty$, then $\phi = f / M$ is a symbol. With $e_n (z) = z^n$, we have, by hypothesis, 
$\phi^n = C_\phi (e_n) \in H^2 (\beta)$. Hence $f^n \in H^2 (\beta)$.  Since it is clear that $f^n \in H^\infty$, we are done. 
\end{proof}

\noindent {\bf Remark.} The conclusion of Proposition~\ref{prop easy} can be obtained with a hypothesis of a different nature. 

\begin{proposition}
If $(\beta_n / n^2)_{n \geq 1}$ is equivalent to a sequence of moments, then $H^2 (\beta) \cap H^\infty$ is an algebra.  
\end{proposition}
\begin{proof}
By hypothesis, $\beta_n \approx n^2 \gamma_n$, where $\gamma_n = \int_0^1 t^{n - 1} \, d\lambda(t)$ and 
where $\lambda$ is a non negative finite measure on  $[0, 1]$. Let $\mu$ be the image of $\lambda$ by the square-root function 
$r \mapsto r^{1/2}$ on $[0, 1]$ and $\nu$ the non negative measure on $\D$ defined by
\begin{displaymath}
\int_{\D} h (z) \, d \nu (z) = \frac{1}{2 \pi} \int_{0}^{2 \pi} \int_{0}^{1} h (r \, \e^{i\theta}) \, d \theta \, d \mu (r) \, .
\end{displaymath}
We have, for $f (z) = \sum_{n = 0}^\infty a_n z^n \in H^{2} (\beta)$:  
\begin{align*}
I (f) 
& := \int_\D | f^{\prime} (z) |^2 \, d \nu(z) = \sum_{n = 1}^\infty n^2 | a_n|^2 \int_0^1 r^{2 n - 2} \, d \mu(r) \\ 
& = \sum_{n = 1}^\infty n^2 | a_n |^2 \int_0^\infty t^{n - 1} \, d \lambda (t)  \approx \sum_{n = 1}^\infty | a_n |^2 \beta_n \, . 
\end{align*} 
Now, we observe that $\| f \|_{H^2 (\beta)}^2 \approx | f (0 ) |^2 + I (f)$ and that 
\begin{displaymath}
I (f^2) = \int_\D 4 \, | f (z) |^2 \, | f^{\prime} (z) |^2 \, d \nu(z) \leq 4 \, \| f \|_\infty^2 \, I (f) \, , 
\end{displaymath}
which clearly gives the result. 
\end{proof} 

A specific example (the Dirichlet space case corresponds to $d \lambda (t) = \omega (t) \, dt = dt$ and $\gamma_n = 1 / n$, or to $\alpha = 0$) is:
\begin{displaymath}
\beta_0 = \beta_1 = 1 \quad \text{and} \quad \beta_n = n\, (\log n)^{\alpha},  \quad 0\leq \alpha\leq 1, \quad \text{for } n \geq 2 \, ,
\end{displaymath}
for which $\sum_{n = 0}^\infty \frac{1}{\beta_n} = + \infty$. Indeed, if  $d \lambda (t) = \omega (t) \, dt$ with
\begin{displaymath}
\omega (t) =\bigg( \log^+ \Big( \frac{1}{\log (1 / t)} \Big) \bigg)^{\alpha} \, , \quad 0 \leq t \leq 1 \, , 
\end{displaymath} 
we have, making the changes of variables $t = \e^{- x}$ and $x = y / n$:
\begin{displaymath}
\gamma_n = \int_0^1 t^{n - 1} \omega (t) \, dt = \frac{1}{n} \int_0^n \e^{- y}\, \bigg( \log \frac{n}{y} \bigg)^{\alpha} \, dy 
\sim \frac{(\log n)^{\alpha}}{n} \, \cdot \qedhere
\end{displaymath}

When $\beta$ is a bounded sequence, it is easy to see that $H^\infty \subseteq H^2 \subseteq H^2 (\beta)$ (see \cite[Proposition~2]{Zorb-PAMS}, for instance). 

\begin{proposition}
Assume that $\beta$ is a non-increasing sequence. Then $H^2 (\beta)$ is an $H^\infty$-module. 
\end{proposition}
\begin{proof}
First of all, since $\beta$ is then bounded, we have $H^\infty \subseteq H^2 (\beta)$. 
Moreover, the shift $S\colon H^2 (\beta) \to H^2 (\beta)$, defined as 
$(S f) (z) = z f (z)$, is a contraction. 
By von Neumann's inequality, we have $\|P (S)\|\leq\|P\|_\infty$ for all polynomials $P$. Since $P(S)f=Pf$, it follows, by approximation, that
\begin{displaymath}
\| g f \|_{H^2 (\beta)} \leq \| g \|_\infty \| f \|_{H^2 (\beta)} 
\end{displaymath}
for all $g \in H^\infty$ and all $f \in H^2 (\beta)$, and that proves that $H^2 (\beta)$ is an $H^\infty$-module. 
\end{proof}

We will now study when $H^2 (\beta)$ is an algebra, and we begin with an elementary remark.

\begin{proposition}
Assume that $H^2 (\beta)$ is an algebra. Then
\begin{enumerate}[(i)]

\item $\sup_{n \geq 0} \beta_{n + 1} / \beta_n < + \infty$. 

\item the sequence $(\beta_n^{1 / n} )_{n \geq 1}$ has a limit. 
\end{enumerate}
\end{proposition}
\begin{proof} The shift, which associates to $f\in H^2(\beta)$ the function $z\mapsto z f (z)$, is bounded from $H^2 (\beta)$ into 
$H^2 (\beta)$. That means that $\sup_{n \geq 0} \beta_{n + 1} / \beta_n < + \infty$ and $(i)$ is proved.

Since $H^2 (\beta)$ is an algebra, there exists a positive constant $C$ such that $\| f g \| \leq C \, \| f \| \, \| g \|$ for all $f, g \in H^2 (\beta)$. Applying that 
to $f (z) = e_m (z) = z^m$ and $g (z) = e_n (z) = z^n$, we get that, for all $m, n \geq 0$:
\begin{equation} \label{beta submultiplicative}
\beta_{m + n} = \| e_{m + n} \| = \| e_m \, e_n \| \leq C \, \| e_m \| \, \| e_n \| = C \, \beta_m \beta_n \, . 
\end{equation}
Setting $\alpha_n = C \, \beta_n$, that means that $\alpha_{m + n} \leq \alpha_m \alpha_n$ for all $m, n \geq 0$, i.e. the sequence $(\alpha_n)_{n \geq 0}$ 
is submultiplicative. 

Now $(ii)$ follows from the multiplicative version of Fekete's lemma (see the proof of the spectral radius formula). 
\end{proof}

\noindent {\bf Remark.} Given $q >1$ and writing $\tilde{\beta}_n = q^n \beta_n$, we have:
\begin{equation}\label{corhomog}
H^2 (\beta)\hbox{ is an algebra if and only if }H^2 (\tilde \beta)\hbox{ is an algebra.}
\end{equation}
Indeed, it is clear that $f\in H^2 (\tilde \beta)$ if and only if $g\in H^2 (\beta)$ where $g(z)=f(\sqrt{q}z)$.
\smallskip 

As a consequence of this remark, we have the following partial answer of the first question raised in the Introduction. 

\begin{theorem} \label{partial answer}
Assume that $H^2 (\beta)$ is an algebra. Let $\phi$ be a symbol which is in $H^2 (\beta)$. Then $\phi$ induces a bounded composition operator 
on $H^2 (\beta)$ in the following cases 
\smallskip

$1)$ when $\liminf_{n \to \infty} \beta_n^{1 / n} > 1$; 
\smallskip

$2)$ when $\liminf_{n \to \infty} \beta_n^{1 / n} = 1$ and $\phi (0) = 0$;
\smallskip

$3)$ when $\beta$ is slowly oscillating. 
\end{theorem}

Recall that the slow oscillation of $\beta$ implies that $\liminf_{n \to \infty} \beta_n^{1 / n} = 1$.
\smallskip

\noindent {\bf Remark.} We will see in Corollary~\ref{cortheo not algebra} that we can have $\liminf_{n \to \infty} \beta_n^{1 / n} = 1$ without 
the boundedness of all the composition operators on $H^2 (\beta)$ with symbol in $H^2 (\beta)$. 

\begin{proof} 
$1)$ Since $\liminf_{n \to \infty} \beta_n^{1 / n} > 1$, the functions in $H^2 (\beta)$ are actually analytic in a disk $D (0, R)$ containing the closed unit disk 
$\overline{\D}$. 
Since $H^2 (\beta)$ is an algebra, we have $H^2 (\beta) = {\mathscr M} \big( H^2 (\beta) \big)$ (see \eqref{prop H2-beta algebra}). By 
\cite[Proposition~20 and Corollary~1 of Proposition~31]{Shields}, the spectrum of each $\phi \in H^2 (\beta)$, as an element of the algebra 
$H^2 (\beta) = {\mathscr M} \big( H^2 (\beta) \big)$ is $\overline{\phi (\D)}$. If $\phi$ is moreover a symbol, we have 
$\overline{\phi (\D)} \subseteq \overline{\D}$. 
By the analytic functional calculus, we have $f (\phi) \in H^2 (\beta)$ for every function $f$ analytic in an open neighborhood of $\overline{\D}$. In particular, 
$f \circ \phi = f (\phi) \in H^2 (\beta)$ for every $f \in H^2 (\beta)$. 
\smallskip

$2)$ Let $f \in H^2 (\beta)$. Let $q > 1$ and set $\tilde \beta_n = q^n \beta_n$; then $\liminf_{n \to \infty} {\tilde \beta_n}^{1 / n} > 1$. We set 
$\tilde \phi (z) = \sqrt{q} \, \phi (z / {\sqrt q})$; we have $\tilde \phi \in H^2 (\tilde \beta)$ and, since $\phi (0) = 0$, $\tilde \phi$ maps $\D$ into $\D$, 
by the Schwarz lemma, so $\tilde \phi$ is a symbol. We set now $g (z) = f (z / {\sqrt q})$; then $g \in H^2 (\tilde \beta)$ and, by $1)$, we 
have $g \circ \tilde \phi \in H^2 (\tilde \beta)$. But 
\begin{displaymath}
g [\tilde \phi (z)] = f [\phi ( z / \sqrt q)] \, , 
\end{displaymath}
and saying that $g \circ \tilde \phi \in H^2 (\tilde \beta)$ is equivalent to saying that $f \circ \phi \in H^2 (\beta)$. Hence $C_\phi$ is bounded on 
$H^2 (\beta)$. 
\smallskip

$3)$ By the parts $1)$ and $2)$, all the symbols $\phi \in H^2 (\beta)$ with $\phi (0) = 0$ induce a bounded composition operator on $H^2 (\beta)$, 
whatever $\beta$. Now, if $\beta$ is slowly oscillating, by \cite[Theorem~4.6]{LLQR-weighted}, all the automorphisms $T_a$ induce bounded composition 
operators on $H^2 (\beta)$. It follows, as in $1)$, because $T_a$ is analytic in an open neighborhood of $\bar \D$, that then all symbols 
$\phi \in H^2 (\beta)$ induce bounded composition operators (if $a = - \phi (0)$, then 
$\psi = T_a \circ \phi = \sum_{n = 0}^\infty \hat{T_a} (n) \phi^n$ is in $H^2 (\beta)$ and $\psi (0) = 0$; the result follows since 
$\phi = T_{- a} \circ \psi$ and $C_\phi = C_\psi \circ C_{T_{- a}}$). 
\end{proof}

The space ${\mathscr M} \big( H^2 (\beta) \big)$ of multipliers of $H^2 (\beta)$ is, by definition, the set of all analytic functions $h$ on $\D$ such that 
$h f \in H^2 (\beta)$ for all $f \in H^2 (\beta)$. It is easy to see (see \cite[beginning of Section~6]{LLQR-weighted}) that 
${\mathscr M} \big( H^2 (\beta) \big) \subseteq H^\infty$.  Actually, we have 
\begin{displaymath}
{\mathscr M} \big( H^2 (\beta) \big) \subseteq H^2 (\beta) \cap H^\infty \, .
\end{displaymath}

Clearly $H^2 (\beta)= {\mathscr M} \big( H^2 (\beta) \big)$ when $H^2 (\beta)$ is an algebra; hence
\begin{equation}\label{prop H2-beta algebra}
\text{If $H^2 (\beta)$ is an algebra, then } H^2 (\beta) \subseteq H^\infty \, ,
\end{equation}
\begin{proposition} \label{theo contained in H-infty} 
We have $H^2 (\beta) \subseteq H^\infty$ if and only if $\sum_{n = 0}^\infty \frac{1}{\beta_n} < + \infty$. 
\end{proposition}
Note that this last condition implies that $H^2 (\beta) \subseteq W^+ (\D)$, where $W^+ (\D)$ is the Wiener algebra of all analytic functions 
$f \colon \D \to \C$ such that $\sum_{n = 0}^\infty | a_n | < + \infty$ if $f (z) = \sum_{n = 0}^\infty a_n z^n$. 
\goodbreak

\begin{proof} 
Assume that $\sum_{n = 0}^\infty \frac{1}{\beta_n} < + \infty$. Then, for $f \in H^2 (\beta)$ and $f (z) = \sum_{n = 0}^\infty a_n z^n$, we have, by the 
Cauchy-Schwarz inequality: 
\begin{displaymath}
\sum_{n = 0}^\infty |a_n| \leq \bigg( \sum_{n = 0}^\infty |a_n|^2 \beta_n \bigg)^{1 / 2} \bigg( \sum_{n = 0}^\infty \frac{1}{\beta_n} \bigg)^{ 1 / 2} 
< + \infty \, . 
\end{displaymath}
Hence $f \in W^+ (\D)$. Therefore $H^2 (\beta) \subseteq W^+ (\D) \subseteq A (\D) \subseteq H^\infty$. 
\smallskip

Conversely, assume that $H^2 (\beta) \subseteq H^\infty$.  Let $(a_n)_{n \geq 0}$ be a sequence such that 
$\sum_{n = 0}^\infty |a_n|^2 \beta_n < + \infty$. Then, setting $f (z) = \sum_{n = 0}^\infty a_n z^n$, we have $f \in H^2 (\beta)$. Now, if 
$g (z) =  \sum_{n = 0}^\infty | a_n | \, z^n$, we have $g \in H^2 (\beta)$ and $\| g \| = \| f \|$. By hypothesis, we hence have $g \in H^\infty$. Then 
\begin{displaymath}
\sup_{| z | < 1} \bigg| \sum_{n = 0}^\infty |a_n| \, z^n \bigg| = \| g \|_\infty 
\end{displaymath}
implies that $\sum_{n = 0}^\infty | a_n | \leq \| g \|_\infty$. We have proved that $\sum_{n = 0}^\infty | a_n | < + \infty$ whenever 
$\sum_{n = 0}^\infty |a_n|^2 \beta_n < + \infty$. That means that $\sum_{n = 0}^\infty \frac{1}{\beta_n} < + \infty$. 
\end{proof}
\begin{corollary} \label{coro CN algebra} 
If $H^2 (\beta)$ is an algebra, then $\sum_{n = 0}^\infty \frac{1}{\beta_n} < + \infty$.
\end{corollary}

This corollary appeared already in \cite[Theorem~3.2.7]{ENZ}. Moreover, \cite[Theorem~3.2.7]{ENZ} states that the converse holds when 
$\beta$ is ``regular'', namely if $\beta$ is log-concave, $\beta_n \geq 1$, $\lim_{n \to \infty} \beta_n^{1 / n} = 1$, and, 
either $\beta_n /n^c$ eventually decreases for some $c > 0$, or either $\beta_n / n^4$ eventually increases and is log-concave. 
However, the converse does not hold in general, as the following result shows. 

\begin{theorem} \label{theo not algebra}
There exists a sequence $\beta$ such that $\beta_{n + 1} / \beta_n \converge_{n \to \infty} 1$, and for which:
\smallskip

$1)$ $\sum_{n = 0}^\infty \frac{1}{\beta_n} < + \infty$; so $H^2 (\beta) \subseteq H^\infty$; 
\smallskip

$2)$ $H^2 (\beta)$ is not an algebra. 
\end{theorem}
\begin{proof}  
Take $m_k = 3^k$ and 
\begin{displaymath}
\left\{ 
\begin{array} {lcl}
\beta_{m_k} & = & m_k^2 \smallskip \\
\beta_{2 m_k} & = & m_k^5 \, , 
\end{array}
\right.
\end{displaymath}
and
\begin{displaymath}
\beta_{n + 1} = \left\{
\begin{array} {ll} 
m_k^{3 / m_k} \beta_n  & \text{for } m_k \leq n \leq 2 m_k - 1 \, ; \smallskip \\ 
9^{1 / m_k} m_k^{- 3 / m_k} \beta_n & \text{for } 2 m_k \leq n \leq m_{k + 1} - 1 \, . 
\end{array}
\right.
\end{displaymath}

Since $\beta_n \gtrsim n^2$, we have $\sum_{n = 0}^\infty \frac{1}{\beta_n} < + \infty$. Since $m_k^{1 / m_k} \converge_{k \to \infty} 1$, we have 
$\beta_{n + 1}/ \beta_n \converge_{n \to \infty} 1$. Moreover, setting $e_n (z) = z^n$, we have 
\begin{displaymath}
\frac{\| e_{m_k}^{\, 2} \|^{\ }}{\| e_{m_k} \|^2} = \frac{ \| e_{2 m_k} \|^{\ }}{\| e_{m_k} \|^2} = \frac{\beta_{2 m_k}^{1 / 2}}{\beta_{m_k}} 
= m_k^{1 / 2} = 3^{k / 2} \converge_{k \to \infty} + \infty \, ;
\end{displaymath}
hence $H^2 (\beta)$ is not an algebra, since otherwise, we would have a positive constant $C$ for which $\| f g \| \leq C \| f \| \, \| g \|$ for all 
$f, g \in H^2 (\beta)$. 
\end{proof}

Theorem~\ref{theo not algebra} allows to answer negatively Zorboska's question 2). 

\begin{theorem}\label{cortheo not algebra}
There exist a sequence $\beta$ such that $\sum_{n = 0}^\infty \frac{1}{\beta_n} < + \infty$ and $\beta_{n + 1} / \beta_n \converge_{n \to \infty} 1$, 
and a symbol $\phi$ such that $\phi \in H^2 (\beta)$ and $\| \phi \|_\infty < 1$, but $C_\phi$ is not bounded on $H^2 (\beta)$.  
\end{theorem}
\begin{proof} 
We use the sequence obtained in Theorem~\ref{theo not algebra}. There exists $\varphi\in  H^2 (\beta)\subset H^\infty$ such that  
$\varphi^2 \notin H^2 (\beta)$; we can assume that $\|\varphi\|_\infty<1$.
Clearly $C_\varphi$ is not bounded on $H^2 (\beta)$ since $z^2\in H^2 (\beta)$ and $C_\phi (z^2) = \phi^2$.
\end{proof}
%

%%%%%%%%%%%%%%%%%%%%%%%%%%%
\subsection{Slowly oscillating sequences} \label{SO}

In Corollary~\ref{coro CN algebra}, we saw that $\sum_{n = 0}^\infty \frac{1}{\beta_n} < \infty$ is necessary for $H^2 (\beta)$ to be an algebra, and 
in \cite[Theorem~3.2.7]{ENZ}, it is shown that the converse holds when $\beta$ satisfies some conditions. Here, we see that the converse holds 
with another regularity condition on $\beta$, namely when $\beta$ is slowly oscillating (recall that we saw in Theorem~\ref{theo CN SO}, 
assuming that  $\lim_{n \to \infty} \beta_n^{1 / n} = 1$, that $\beta$ slowly oscillating is necessary in order that every symbol 
$\phi \in H^2 (\beta)$ induces a bounded composition operator on $H^2 (\beta)$). 

\begin{theorem} \label{theo converse with SO}
If $\sum_{n = 0}^\infty \frac{1}{\beta_n} < + \infty$ and $\beta$ is slowly oscillating, then $H^2 (\beta)$ is an algebra. 
\end{theorem}

Using Theorem~\ref{partial answer}, we get the following corollary.

\begin{theorem}
If $\sum_{n = 0}^\infty \frac{1}{\beta_n} < + \infty$ and $\beta$ is slowly oscillating, then all symbols $\phi \in H^2 (\beta)$ induce bounded 
composition operators on $H^2 (\beta)$. 
\end{theorem}

Theorem~\ref{theo converse with SO} can be deduced from the condition (3.6.1) of \cite[Lemma~3.6.3]{ENZ}, but we give a direct proof for 
sake of completeness. We introduce the following notation:

\begin{equation} \label{BN} 
B_n = \sum_{k = 0}^n \frac{1}{\beta_k \beta_{n - k}} \, \cdot
\end{equation}

Note that $B_n \geq \frac{1}{\beta_0 \beta_n}$ so $\inf_{n \geq 0} B_n \beta_n > 0$. 
\medskip

The proof will follow from the next two lemmas. 

\begin{lemma}
If $\sum_{n = 0}^\infty \frac{1}{\beta_n} < + \infty$ and $\beta$ is slowly oscillating, then $\beta_n B_n = O (1)$. 
\end{lemma}
\begin{proof}
Set $M = \sum_{n = 0}^\infty \frac{1}{\beta_n} \,$. We have 
\begin{displaymath}
B_n \leq 2 \sum_{0 \leq k \leq n / 2} \frac{1}{\beta_k \beta_{n - k}} \, \cdot
\end{displaymath}
But $\frac{1}{\beta_{n - k}} \leq C \, \frac{1}{\beta_n} \,$, since $n/ 2 \leq n - k \leq n$ for $0 \leq k \leq n / 2$, and $\beta$ is  slowly oscillating. Hence 
\begin{displaymath}
B_n \leq 2 \, C \, \frac{1}{\beta_n} \sum_{0 \leq k \leq n / 2} \frac{1}{\beta_k} \leq 2 \, C M \, \frac{1}{\beta_n} \, \cdot \qedhere
\end{displaymath}
\end{proof}

The second lemma is nothing else than the condition (3.2.2) of \cite{ENZ} (see also \cite[Lemma, page~207]{N1}). 

\begin{lemma} \label{Pascal criterion} 
If $\beta_n B_n = O (1)$, then $H^2 (\beta)$ is an algebra. 
\end{lemma}
\begin{proof}
It suffices to show that $f^2 \in H^2 (\beta)$ for all $f \in H^2 (\beta)$. Let $f \in H^2 (\beta)$ and write $f (z) = \sum_{n = 0}^\infty a_n z^n$. We have 
\begin{displaymath}
[ f (z)]^2 = \sum_{n = 0}^\infty \bigg( \sum_{k = 0}^n a_k a_{n - k} \bigg) \, z^n \, . 
\end{displaymath}
By the Cauchy-Schwarz inequality, we have
\begin{align*}
\sum_{n = 0}^\infty \bigg|  \sum_{k = 0}^n a_k a_{n - k} \bigg|^2 \beta_n 
& \leq \sum_{n = 0}^\infty \bigg( \sum_{k = 0}^n | a_k a_{n - k} |^2 \beta_k \beta_{n - k} \bigg) \, 
\bigg( \sum_{k = 0}^n \frac{1}{\beta_k \beta_{n - k}} \bigg) \, \beta_n \\ 
& \lesssim \sum_{n = 0}^\infty \bigg( \sum_{k = 0}^n | a_k a_{n - k} |^2 \beta_k \beta_{n - k} \bigg) \\ 
& = \bigg( \sum_{j = 0}^\infty | a_j |^2 \beta_j \bigg)^2 = \| f \|^4 < + \infty \, , 
\end{align*}
which says that $f^2 \in H^2 (\beta)$. 
\end{proof}

As said, when $\beta$ is ``regular'', $\sum_{n = 0}^\infty \frac{1}{\beta_n} < + \infty$ is also necessary in order that $H^2 (\beta)$ be an algebra 
(\cite[Theorem~3.2.7]{ENZ}).  
In general, we have no necessary and sufficient condition. However, we can state the following necessary conditions. 

\begin{proposition} \label{prop CN algebra} 
If $H^2 (\beta)$ is an algebra, then 
\smallskip

$1)$ $\beta_n B_n = O (n)$; 
\smallskip

$2)$ $\sum_{n = 0}^\infty \beta_n B_n^2 < + \infty$. 
\end{proposition}
\begin{proof}
1) By \eqref{beta submultiplicative}, we have $\frac{1}{\beta_k \beta_l} \leq C \, \frac{1}{\beta_{k + l}}$\,; so 
\begin{displaymath}
B_n = \sum_{k + l = n} \frac{1}{\beta_k \beta_l} \leq C n \, \frac{1}{\beta_n} \, \cdot
\end{displaymath}

2) Let $ f (z) = \sum_{n = 0}^\infty \frac{1}{\beta_n} \, z^n$; we have $\| f \|^2 = \sum_{n = 0}^\infty \frac{1}{\beta_n} < + \infty$, so 
$f \in H^2 (\beta)$. Since $H^2 (\beta)$ is an algebra, we have $f^2 \in H^2 (\beta)$. But $[ f (z)]^2 = \sum_{n = 0}^\infty B_n z^n$, so we get 
$\sum_{n = 0}^\infty B_n^2 \beta_n < + \infty$. 
\end{proof}

The necessary conditions of Corollary~\ref{coro CN algebra} and of Proposition~\ref{prop CN algebra} are not sufficient. 

\begin{proposition} \label{prop not CS for algebra}
There exists a sequence $\beta$ such that $\sum_{n = 0}^\infty \beta_n B_n^2 < + \infty$, hence  $\sum_{n = 0}^\infty \frac{1}{\beta_n} < + \infty$, 
but for which $H^2 (\beta)$ is not an algebra. 
\end{proposition}
\begin{proof}
We take $\beta_0 = 1$, and, for $n \geq 1$, $\beta_n = n^\gamma$ for $n$ even and $\beta_n = n^{\gamma'}$ for $n$ odd, where 
$1 < \gamma ' < \gamma$ and $2 \gamma' > \gamma + 1$ (for example, $\gamma' = 3$ and $\gamma = 4$). Since $\beta_{n + 1} / \beta_n$ is 
not bounded, $H^2 (\beta)$ is not an algebra. 
It is clear that $\sum_{n = 0}^\infty \frac{1}{\beta_n} < + \infty$. 

Now
\begin{align*}
B_{2 n} 
& = \sum_{k = 0}^{2 n} \frac{1}{\beta_k \beta_{2 n - k}}  
\approx \sum_{k \text{ even}} \frac{1}{k^\gamma (2 n - k)^\gamma}  + \sum_{k \text{ odd}} \frac{1}{k^{\gamma'} (2 n - k)^{\gamma '}} \\
& \lesssim \frac{1}{n^\gamma} + \frac{1}{n^{\gamma '}} \lesssim  \frac{1}{n^{\gamma '}} \, \raise 1 pt \hbox{,} 
\end{align*}
so $(2 n)^\gamma B_{2 n}^2 \lesssim (2 n)^{\gamma - 2 \gamma '}$. 
\smallskip

Similarly $(2 n + 1)^{\gamma '} B_{2 n + 1}^2 \lesssim (2 n + 1)^{- \gamma '}$. 
\smallskip

Since $\gamma - 2 \gamma ' < - 1$ and $\gamma '  > 1$, we obtain the convergence of the series $\sum_{n \geq 0}^\infty \beta_n B_n^2$. 
\end{proof}

Nevertheless, $\beta$ slowly oscillating is not necessary for $H^2 (\beta)$ to be an algebra. 

\begin{example} \label{theo exp sqrt n}
If $\beta_n = \e^{\sqrt n}$ for all $n \geq 0$, then $H^2 (\beta)$ is an algebra, though $\beta$ is not slowly oscillating. 
\end{example}

That $\beta$ is not slowly oscillating is clear. That $H^2 (\beta)$ is an algebra, follows from \cite[Lemma~3.5.5]{ENZ}.  

As a consequence, using Theorem~\ref{theo CN SO}, we get the following result. 

\begin{theorem}
There exist a sequence $\beta$ such that $H^2 (\beta)$ is an algebra, and a symbol $\phi \in H^2 (\beta)$ for which $C_\phi$ is not bounded on 
$H^2 (\beta)$. 
\end{theorem}
%

%%%%%%%%%%%%%%%%
\subsection{Hankel matrices} \label{matrices} 

Actually the problem to know whether $H^2(\beta)$ is an algebra or not can be formulated in terms of Schur multipliers acting on a family of Hankel matrices.

\begin{theorem}\label{schuralgebra}

$H^2(\beta)$ is an algebra if and only if the map
\begin{displaymath}
\begin{array}{cccl}
\Psi \colon & \ell^2 & \longrightarrow & B(\ell^2) \cr
& u & \longmapsto & \Bigg( {\displaystyle u_{k+l}\sqrt{\frac{\beta_{k+l}}{\beta_k\beta_{l}}} } \Bigg)_{k, l}
\end{array}
\end{displaymath}
is bounded.
\end{theorem}

Let us point out that Lemma~\ref{Pascal criterion} also follows from this theorem. Actually, the condition $\beta_n B_n = O (1)$ is equivalent to the 
fact that $\Psi$ is bounded from $\ell^2$ to $HS (\ell^2) \subseteq B (\ell^2)$, where $HS (\ell^2)$ stands for the space of Hilbert-Schmidt operators on 
$\ell^2$.

Indeed, for every $u\in\ell^2$, we have $\dis\|\Psi(u)\|_{HS}^2=\sum_n |u_n|^2 \beta_n B_n$\,. 

This remark gives a hint on the gap between the sufficient condition $\beta_n B_n = O (1)$ and a potential characterization.

\begin{proof}[Proof of Theorem \ref{schuralgebra}]

The vector space $H^2 (\beta)$ is an algebra if and only if $fg \in H^2 (\beta)$ for all $f,g \in H^2 (\beta)$. In other words,  $H^2 (\beta)$ is an algebra 
if and only if  
\begin{displaymath}
\sum_{n = 0}^\infty \bigg| \sum_{p = 0}^n a_p b_{n - p} \bigg|^2 \beta_n < + \infty 
\end{displaymath}
whenever $\sum_{n = 0}^\infty |a_n |^2 \beta_n < + \infty$ and $\sum_{n = 0}^\infty |b_n |^2 \beta_n < + \infty$. Equivalently,% that means that 

\begin{displaymath}
\biindice{\sup}{a,b \in \ell^2}{\|a\|=\|b\|=1} 
\sum_{n = 0}^\infty \Bigg| \sum_{k+l=n}\frac{a_k}{\sqrt{\beta_k}}\frac{ b_l}{\sqrt{\beta_{l}}}\sqrt{\beta_{k+l}} \Bigg|^2  < + \infty 
\end{displaymath}
or, in other words,
\begin{displaymath}
\biindice{\sup}{a,b,u \in \ell^2}{\|a\|=\|b\|=\|u\|=1} 
\Bigg| \sum_{n = 0}^\infty \sum_{k+l=n} a_k b_{l} u_n \sqrt{\frac{\beta_{k+l}}{\beta_k\beta_{l}}} \Bigg|  < + \infty \,  . 
\end{displaymath}

Let us point out that 
$$\sum_{n = 0}^\infty \sum_{k+l=n} a_k b_{l}u_n\sqrt{\frac{\beta_{k+l}}{\beta_k\beta_l}} 
= \sum_{k = 0}^\infty a_k \sum_{l = 0}^\infty M_{k,l} b_{l}=\sum_{k = 0}^\infty a_k \Big(\Psi(u)(b)\Big)_k$$
where $M_{k,l}=\dis u_{k+l}\sqrt{\frac{\beta_{k+l}}{\beta_k\beta_{l}}} \, \cdot$

Therefore,  $H^2 (\beta)$ is an algebra if and only if% for every $u\in\ell^2$, 
\begin{displaymath}
\biindice{\sup}{b,u \in \ell^2}{\|b\|=\|u\|=1} \Big\| \Psi(u) (b) \Big\| < + \infty \, . 
\end{displaymath}
and the result follows. 
\end{proof}

As an application of Theorem~\ref{schuralgebra}, we see that the equivalence between ``$H^2 (\beta)$ is an algebra'' and 
$\sum_{n = 0}^\infty \frac{1}{\beta_n} < + \infty$ does not follows from the submultiplicativity of $\beta$. 

\begin{theorem} \label{submultiplicative}
\hfill There \hfill exists \hfill a \hfill submultiplicative \hfill sequence \hfill $\beta$ \hfill such \hfill that $\lim_{n \to \infty} \beta_n^{1 / n} = 1$ and 
$\sum_{n = 0}^\infty \frac{1}{\beta_n} < + \infty$, and for which $H^2 (\beta)$ is not an algebra. 
\end{theorem}
\begin{proof} 
Let $m_k = 3^k$ and $\sigma (n)$ defined by $\sigma (0) = 1$ and 
\begin{displaymath}
\sigma (n) = k \quad \text{for } m_{k - 1} \leq n < m_k \text{ and } k  \geq 1 \, .
 \end{displaymath}
We define 
\begin{displaymath}
\beta_n = \exp (n / \sigma (n)) \, . 
\end{displaymath}

$\bullet$ The submultiplicativity of $\beta$ is obvious since $\sigma$ is nondecreasing:
\begin{displaymath}
\frac{m + n}{\sigma (m + n)} = \frac{m}{\sigma (m + n)} + \frac{n}{\sigma (m + n)} 
\leq \frac{m}{\sigma (m)}+ \frac{n}{\sigma (n)} \, \cdot
\end{displaymath}

$\bullet$ That $\lim_{n \to \infty} \beta_n^{1 / n} = 1$ is clear. 
\smallskip

$\bullet$ That $\sum_{n = 0}^\infty \frac{1}{\beta_n} < + \infty$ is easy: since $m_k-m_{k-1}=2m_k/3$, we get using geometric series:
\begin{align*}
\sum_{n = 0}^\infty \exp [- n / \sigma (n)] 
& = \sum_{k = 1}^\infty \bigg( \sum_{m_{k - 1} \leq n < m_k}  \exp (- n / k) \bigg) 
\lesssim \sum_{k = 1}^\infty \frac{\exp (- 2 m_k / 3)}{1  - \e^{- 1 / k}} \\ 
& \lesssim \sum_{k = 1}^\infty k \exp (- 2 m_k / 3) < + \infty \, .
\end{align*}

$\bullet$ We give two proofs that $H^2 (\beta)$ is not an algebra. 
\smallskip

-- \emph{First proof}. Let $A_n$ be the $(n \times n)$ self-adjoint matrix with all entries equal to $1$, which satisfies  $A_{n}^2 = n A_n$. Hence 
$\Vert A_n \Vert^2 = \Vert A_{n}^2 \Vert = n \Vert A_n\Vert$ and so $\Vert A_n \Vert = n$. 

Let us be more specific. Consider the interval of integers $I_k = [m_k / 3, \, m_k / 2)$ and let $u = \ind_{2 I_k} =\ind_{[2 m_k / 3, \, m_k)} \in \ell^2$. 
We have $\Vert u \Vert_{\ell^2} \approx \sqrt{m_k}$. 

Now, let $\Psi (u)$ be the matrix $\Big(u_{i + j} \sqrt{\frac{\beta_{i + j}}{\beta_i \beta_j}}\Big)_{i, j}$. We recall that, with obvious notations: 
\begin{displaymath}
| a_{i, j} | \leq b_{i, j} \quad \Longrightarrow \quad \Vert A \Vert \leq \Vert B \Vert \, .
\end{displaymath}
Therefore, since $i, j \in I_k \Rightarrow i + j \in 2I_k$:
\begin{align*}
\Vert \Psi (u) \Vert 
& \geq \bigg\| \bigg( u_{i + j} \sqrt{\frac{\beta_{i + j}}{\beta_i \beta_j}} \bigg)_{i, j \in I_k} \bigg\|
= \bigg\| \bigg( \sqrt{\frac{\beta_{i + j}}{\beta_i \beta_j}} \bigg)_{i, j \in I_k} \bigg\|\\ 
\smallskip
&  = \| (1 )_{i, j \in I_k} \| = |I_k| \approx m_k \, .
\end{align*} 
Indeed, for $i, j \in I_k$, we have $i, j, i + j \in [m_k /3, m_k) = [m_{k - 1}, m_k)$, so that
\begin{displaymath}
\beta_{i + j} = \exp [(i + j) / k] = \exp(i / k) \exp (j / k) = \beta_i \beta_j \, ,
\end{displaymath}
and we use the norm of the matrix $A_{|I_k|}$.

Hence the quotient $\frac{\Vert \Psi (u) \Vert_{B (\ell^2)}}{\Vert u \Vert_{\ell^2}}$ is unbounded as $k \to \infty$.
\smallskip

-- \emph{Second proof}. Actually, this proof was obtained thanks to the first one. Let $f = f_k$ be the function defined by
\begin{displaymath}
f (z) = \sum_{m_k /  3 \leq n < m_k /  2} z^n = \sum_{n \in I_k} z^n \, .
\end{displaymath}
We are going to prove that 
\begin{equation}\label{ecart} 
\frac{\Vert f^2 \Vert \ }{\Vert f \Vert^{2}} \converge_{k \to \infty} + \infty \, .
\end{equation}
For that, we first prove that
\begin{equation}\label{majo} 
\Vert f \Vert^2 \lesssim k \exp \Big(\frac{m_k}{2 k} \Big) \, \cdot
\end{equation}
Indeed, setting again $I_k = [m_k / 3, \, m_k / 2)$ and using $1/3 + 1/6 = 1/2$:
\begin{displaymath}
\Vert f \Vert^2 = \sum_{n \in I_k} \exp \Big( \frac{n}{k} \Big) = \e^{m_k / 3 k} \sum_{0 \leq j \leq m_k / 6} \e^{j / k} 
\lesssim \frac{\exp (\frac{m_k}{2 k})}{\exp (\frac{1}{k}) -  1} \lesssim k \exp \Big( \frac{m_k}{2 k} \Big) \, \cdot
\end{displaymath}

It remains to show that
\begin{equation}\label{mino} 
\Vert f^2 \Vert \gtrsim k^{3/2} \exp \Big(\frac{m_k}{2 k} \Big) \, \cdot
\end{equation}
For that, setting $J_k = [2 m_k / 3, m_k) = 2 I_k$, we have $f^2 = \sum_{n \in J_k} c_n z^n$, and 
\begin{displaymath}
\Vert f^2 \Vert^{2} = \sum_{n \in J_k} c_{n}^2 \exp \Big(\frac{n}{k} \Big) \, ,
\end{displaymath}
where $c_n$ is the number of indices $i \in I_k$ such that $n - i \in I_k$, that is 
\begin{displaymath}
\max (m_k / 3, \, n - m_k / 2) \leq i \leq \min (m_k / 2, \, n - m_k / 3) \, . 
\end{displaymath}
When $(5/6) m_k \leq n  < m_k$, this amounts to $n - m_k / 2 \leq i < m_k / 2$, implying $c_n  \geq m_k - n$. 

Setting $q = \e^{1/k}$ and observing that $m_{k}^{2} \, q^{- m_k / 6} \converge_{k \to \infty} 0$, we get
\begin{displaymath}
\Vert f^2 \Vert^2 \geq \sum_{5/6 m_k \leq n < m_k} (m_k  - n)^2 q^{n} = q^{m_k} \sum_{0 \leq j \leq m_k / 6} j^2 q^{- j} 
\approx q^{m_k} (1 - q^{- 1})^{- 3} \, , 
\end{displaymath}
that is $\Vert f^2 \Vert^2 \gtrsim k^3 \exp (m_k / k)$, since $(1 - q^{- 1})^{- 1} \sim k$, and that proves \eqref{mino}. 
\end{proof}

\noindent {\bf Remark.} Let us observe that, in the above proof, $\sigma (n)$ behaves, roughly speaking, as $\log n$, so that $\beta_n$ behaves as 
$\gamma_n = \exp (n / \log n)$, although $H^{2}(\gamma)$ is an algebra, as said in the next proposition. The difference between 
$\beta_n = \exp [n /\sigma (n)]$, and $\gamma_n$, can be heuristically explained as follows. Set $\tau (x) = \frac{x}{\log x}$ for $x > 1$ and 
$\tilde \tau_n = n / \sigma (n)$. 
In the discrete case, we have 
\begin{displaymath}
j, n \in [m_k / 3,\, m_k / 2] \quad \Longrightarrow \quad \tilde\tau (n + j) - \tilde\tau (n) - \tilde\tau (j) = 0 \, , 
\end{displaymath}
meaning that ${\tilde\tau}{\, ''} = 0$, while in the continuous case we have 
\begin{displaymath}
\tau'' (x) \sim \frac{- 1}{x (\log x)^2} < 0 \, .
\end{displaymath}

Proposition~\ref{prop n/ log n} follows from \cite[Theorem~3.2.7]{ENZ}, but, for sake of completeness, we prefer to give 
a short elementary proof of this particular case. 

\begin{proposition} \label{prop n/ log n}
For $\gamma (n) = \e^{n / \log n}$, $n \geq 2$ {\rm (}and $\gamma_0 = \gamma_1 = 1$, for instance{\rm )}, the space $H^2 (\gamma)$ is an algebra. 
\end{proposition}
\begin{proof}
Let 
\begin{displaymath}
\alpha_{n, k} = \frac{n}{\log n} - \frac{k}{\log k} - \frac{n - k}{\log (n - k)} \, \cdot
\end{displaymath}

We have to prove, using Lemma~\ref{Pascal criterion}, that
\begin{displaymath}
\sup_n \bigg( \sum_{k = 2}^{n / 2} \exp (\alpha_{n, k}) \bigg) < + \infty \, . 
\end{displaymath}

Observe that, for $2 \leq k \leq \sqrt n$, we have
\begin{displaymath}
\alpha_{n, k}  \leq \frac{n}{\log n} - \frac{k}{\log k} - \frac{n - k}{\log n} = k \bigg( \frac{1}{\log n} - \frac{1}{\log k} \bigg) 
\leq - \frac{k}{2 \log k} \, ;
\end{displaymath}
hence
\begin{displaymath}
\sum_{k = 2}^{\sqrt n} \exp (\alpha_{n, k}) \leq \sum_{k = 2}^\infty \exp \bigg( - \frac{k}{2 \log k} \bigg) := S < + \infty \, .
\end{displaymath}

For $\sqrt n \leq k \leq n / 2$, we have, in the same way:
\begin{align*}
\alpha_{n, k} 
& \leq k \bigg( \frac{1}{\log n} - \frac{1}{\log k} \bigg) \leq \sqrt n \bigg( \frac{1}{\log n} - \frac{1}{\log n - \log 2} \bigg) \\ 
& = - \frac{{\sqrt n} \, \log 2}{ (\log n) (\log n - \log 2)} \, ;
\end{align*}
hence 
\begin{displaymath}
\sum_{\sqrt n \leq k \leq n / 2} \exp (\alpha_{n, k}) \leq \frac{n}{2} \exp \bigg( -  \frac{{\sqrt n} \, \log 2}{ (\log n) (\log n - \log 2)} \bigg) 
\converge_{n \to \infty} 0 \, . \qedhere
\end{displaymath}
\end{proof}
%

%%%%%%%%%%%%%%%%%%%%%%%%%%%%%%%%%%%%%%%%%%%%%%%%%%%%%%%%%%%%%%%%%%%%%%%%%%
\section{Composition operators induced by the automorphisms on $h^p (\beta)$} \label{sec hp}

For $p=2$, it is easy to check that all the composition operators  $C_{T_a}$, for  $a \in \D$, generated by the automorphisms $T_a$, defined in 
\eqref{def T_a}, (and actually all composition operators, by Littlewood's subordination principle) are  bounded on $H^2$. We recently studied 
\cite{LLQR-weighted} the weighted version $H^{2}(\beta)$ of this space and gave a complete characterization of those weights $\beta = (\beta_n)$ for which 
either $C_\varphi$ is bounded on $H^{2} (\beta)$ for all automorphisms, or for all symbols. In this section, we show that this is never the case for $p \neq 2$. 
We mention in passing that this boundedness issue was previously considered by Blyudze and Shimorin \cite{BlSh} in the case when the initial space is $h^1$ 
and $h^p$ is the target space. 
In that case, the authors show that $C_{T_a} \colon h^1 \to h^p$ is bounded if and only if $p \geq 2$. This has been made much more precise in \cite{SZZA}. 

We will prove the  following result. 
\begin{theorem} \label{alenvers}  
Let $p\in[1,+\infty]$, with $p \neq 2$. Then $C_{T_a}$ is unbounded on $h^p$ for all $a \in \D \setminus \{0\}$. 

Moreover, $C_{T_a}$ is never bounded on $h^{p}(\beta)$, whatever the choice of $\beta$ satisfying \eqref{condition on beta}.
\end{theorem} 

S.~Charpentier, N.~Espouiller and R.~Zarouf informed us that they proved, independently, Theorem~\ref{alenvers} for $\beta \equiv 1$, using \cite{SZZA} 
(see \cite{CEZ}). 

\begin{proof} Set 
\begin{displaymath}
I = \bigg[\frac{1}{2} \, \raise 1,5 pt \hbox{,} \, \frac{2}{3} \bigg] \quad \text{ and } \quad  J  = [\alpha^{-  1}, \alpha] \, ,
\end{displaymath}
($\alpha= 5/4$ for instance).

As in \cite[Proposition 4.2]{LLQR-weighted}, the rotation invariance of $h^{p} (\beta)$ allows us to claim that, if $C_{T_a}$ is bounded on $h^{p} (\beta)$ 
for some $a \in \D$, $a \neq 0$, then there exists a constant $K$ such that
\begin{equation} \label{M} 
\Vert C_{T_a} \Vert \leq K \, , \quad  \forall a \in I \, . 
\end{equation} 

Now, everything will rely on the matrix $A = (a_{m,n})_{m, n}$, where 
\begin{displaymath}
a_{m,n} = \widehat{(T_a)^{n}} (m) \,\bigg( \frac{\beta_m}{\beta_n} \bigg)^{1/p} \, ,
\end{displaymath}
which represents $C_{T_a}$ on the canonical (Schauder) basis of $h^{p} (\beta)$. If this matrix defines a bounded operator, its columns and  rows  
(the columns of the transposed operator) are respectively uniformly bounded on $\ell^p$ and on the dual space $\ell^q$ ($q$ the conjugate exponent of $p$), 
that is  (with some $R$ independent of $a \in I$):
\begin{equation} \label{colonne} 
C_n = \sum_{m = 0}^\infty |a_{m, n}|^p \, \frac{\beta_m}{\beta_n} \quad\qquad \text{satisfies} \quad \sup_{n} C_n  \leq R < + \infty \, ,
\end{equation} 
and
\begin{equation} \label{ligne} 
L_m = \sum_{n = 0}^\infty |a_{m, n}|^q \, \bigg( \frac{\beta_m}{\beta_n} \bigg)^{q / p} \quad \text{satisfies} \quad \sup_{m} L_m \leq R < + \infty \, .
\end{equation} 

We will show that, for $p \neq 2$, one of the necessary conditions \eqref{colonne} or  \eqref{ligne} fails. We will hence separate two cases. 
\smallskip

We need an auxiliary result. Recall first the following elementary lemma (see  \cite[Lemma 4.16]{LLQR-weighted}). 
\begin{lemma} \label{rappel} 
Let $f\colon [A, B] \to \R$ be a ${\cal C}^2$-function such that $| f ' |\geq \delta$ and $| f '' | \leq M$. Then
\begin{displaymath} 
\bigg| \int_{A}^B \e^{i f (x)} \, dx \bigg| \leq \frac{2}{\delta} + \frac{M (B - A)}{\delta^2} \, \cdot
\end{displaymath}
\end{lemma}

This lemma implies the following extended version of  \cite[Proposition 4.12]{LLQR-weighted}. 
\begin{proposition} \label{new} 
If $r \in J = [\alpha^{- 1}, \alpha]$ and $s \geq 1$, it holds, when $r = m / n$:
\begin{equation} \label{hold} 
\int_{I} | \widehat{T_{a}^{n}} (m) |^{s} \, da \gtrsim \bigg( \int_{I} | \widehat{T_{a}^{n}} (m) | \, da \bigg)^{s} \geq \delta \, n^{- s / 2} \, .
\end{equation} 
\end{proposition}
\begin{proof} The first inequality in \eqref{hold} is just H\"older's inequality. For the second one, we proved in \cite[Proposition 4.14]{LLQR-weighted} that 
\begin{equation} \label{puisq} 
|\widehat{T_{a}^{n}} (m) | \geq \delta \, n^{- 1 / 2} \big| \cos (n \psi_{r} (a) + \pi / 4 ) \big| + O \, (n^{- 3 / 5}) \, ,
\end{equation} 
where $\psi_r = f$ satisfies the assumptions of Lemma~\ref{rappel}, and the constant controlling the error term $O \, (n^{- 3 / 5})$ 
\emph{does not depend on $a$}. Moreover, we have the classical Fourier expansion 
\begin{displaymath}
| \cos x | = c + \sum_{l = 1}^\infty \delta_{l} \cos l x \, , 
\end{displaymath}
with $c > 0$ and $\delta_l = O \, ( l^{- 2})$. Hence (actually $\sum_{l = 1}^\infty | \delta_l | < + \infty$ would suffice): 
\begin{equation} \label{expans}  
|\widehat{T_{a}^{n}} (m) | 
\geq \delta \, n^{- 1 / 2} \bigg( c + \sum_{l = 1}^\infty  \delta_l \, \cos \big(l (n\psi_{r} (a) + \pi / 4 ) \big) \bigg) + O\, (n^{- 3 / 5}) \, .
\end{equation}
We now apply Lemma~\ref{rappel} with $f (a) = l \big(n \psi_{r} (a) + \pi / 4 \big)$. Here, for given $l$, we have $|f ' | \geq \delta \, n l$ and 
$|f '' |\leq  M n l$ on $I$, so that
\begin{displaymath}
\bigg| \int_{I} \cos \big(  l (n\psi_{r} (a) + \pi / 4 ) \big) \, da \bigg| \lesssim \bigg( \frac{1}{n l} + \frac{n l}{n^2 l^2} \bigg) 
\approx \frac{1}{n l} \, \cdot
\end{displaymath}
It now follows from \eqref{puisq} that 
\begin{align*}
\int_{I} | \widehat{T_{a}^{n}} (m) | \, da 
& \geq \delta \, n^{- 1 / 2} \bigg( c + O \, \Big( \sum_{l = 1}^\infty \frac{| \delta_l |}{n l} \Big) \bigg) + O \, (n^{- 3 / 5}) \\ 
& \geq \delta c\, n^{- 1 / 2} + O \, (n^{- 3 / 5}) \geq  \delta' \, n^{- 1 / 2} \, ,
\end{align*}
and this ends the proof.
\end{proof}

We now come back to the proof of Theorem~\ref{alenvers}. We will reason by contradiction, and separate two cases. 
\smallskip

We set $J_l = [l/\sqrt{\alpha}\,, \sqrt{\alpha}~ l]$. 
\smallskip

\noindent $\bullet$ {\sl Case $p < 2$}. 

It follows from \eqref{colonne} that,  for $n \in J_l$, we have
\begin{displaymath}
\sum_{m \in J_l}  | \widehat{T_{a}^{n}} (m) |^p \, \frac{\beta_m}{\beta_n} \leq R \, .
\end{displaymath}
Integrating on $I$ and using Proposition~\ref{new} give, since $m, n \in J_l$ (hence $m/n\in J$) have the same size as $l$: 
\begin{displaymath}
l^{- p / 2} \sum_{m \in J_l} \frac{\beta_m}{\beta_n} \lesssim R  \, .
\end{displaymath}
Now, summing up over $n \in J_l$ further gives 
\begin{displaymath}
l^{- p / 2} \bigg( \sum_{k \in J_l} \beta_k \bigg) \bigg(\sum_{k \in J_l} \beta_{k}^{- 1} \bigg) \lesssim R\, l \, .
\end{displaymath}
But, by the Cauchy-Schwarz inequality:
\begin{displaymath}
l^2 \lesssim |J_l|^2 \leq \bigg( \sum_{k \in J_l} \beta_k \bigg) \bigg( \sum_{k \in J_l} \beta_{k}^{-  1} \bigg) \, ,
\end{displaymath}
and we get $l^{2 - p / 2} \lesssim R \, l$ or $R \gtrsim l^{1 - p / 2}$. 

Since $1 - p / 2 > 0$, we have a contradiction for large $l$. 
\smallskip

\noindent $\bullet$ {\sl Case $p > 2$}. 

The proof is nearly the same. It now follows from \eqref{ligne} that, for $m \in J_l$:
\begin{displaymath}
\sum_{n \in J_l}  | \widehat{T_{a}^{n}} (m)  |^q \, \bigg( \frac{\beta_m}{\beta_n} \bigg)^{q / p} \leq R \, .
\end{displaymath}
Set $\gamma_k = \beta_{k}^{\, q / p}$  and proceed as in the case $p < 2$ to get 
\begin{displaymath}
l^{- q / 2} \bigg( \sum_{k \in J_l} \gamma_k \bigg) \bigg( \sum_{k \in J_l} \gamma_{k}^{- 1}\bigg) \lesssim R \, l \, ,
\end{displaymath}
or, again by the Cauchy-Schwarz inequality, $l^{2 - q / 2} \lesssim R \, l$ or $R \gtrsim l^{1 -q  / 2}$. 

Since this time $1 - q / 2 > 0$, we have again a contradiction for large $l$. 
\end{proof}

\noindent{\bf Acknowledgement.}  Parts of this paper was written while the fourth-named author was visiting the Universit\'e de Lille and the 
Universit\'e d'Artois in January 2023 and while the first-named author was visiting the Universidad de Sevilla in March 2023; it a pleasure to thank all their 
colleagues of these universities for their warm welcome. 

H.~Queff\'elec is partially supported by the Labex CEMPI, ANR-11-LABX-0007-01.

L.~Rodr{\'\i}guez-Piazza research has been funded by the PID2022-136320NB-I00 project / AEI/10.13039 /501100011033/ FEDER, UE. 

%%%%%%%%%%%%%%%%%%%%%%%%%%%%%%%%%%%%%%%%%%%%%%%%%%%%%%%%%%%%%%%%%%%%%%%%%%

%%%%%%%%%%%%%%%%%%%%%%%%%%%%%%%%%%%%%%%%%%%%%%%%%%%%%%%%%%%%%%%%%%%%%%%%%%%%%%%
\smallskip\goodbreak

{\footnotesize
Pascal Lef\`evre \\
Univ. Artois, UR 2462, Laboratoire de Math\'ematiques de Lens (LML), F-62300 Lens, France \\
pascal.lefevre@univ-artois.fr
\smallskip

Daniel Li \\ 
Univ. Artois, UR 2462, Laboratoire de Math\'ematiques de Lens (LML), F-62300 Lens, France \\
daniel.li@univ-artois.fr
\smallskip

Herv\'e Queff\'elec \\
Univ. Lille, CNRS, UMR 8524 -- Laboratoire Paul Painlev\'e, F-59000 Lille, France \\
Herve.Queffelec@univ-lille.fr
\smallskip
 
Luis Rodr{\'\i}guez-Piazza \\
Universidad de Sevilla, Facultad de Matem\'aticas, Departamento de An\'alisis Matem\'atico \& IMUS,  
Calle Tarfia s/n  
41\kern 1mm 012 SEVILLA, SPAIN \\
piazza@us.es
}

\end{document}